\documentclass{scrartcl}

\usepackage{misere,  e-jc}

\dateline{}{}{}
\title{The Sprague-Grundy Functions of Saturations of Mis\`{e}re Nim}
\author{Yuki Irie\thanks{This work was partially carried out at Chiba University.}\\
\small  Research Alliance Center for Mathematical Sciences\\[-0.8ex]
\small Tohoku University\\[-0.8ex] 
\small Miyagi, Japan\\
\small\texttt {yirie@tohoku.ac.jp}\\
}
\Copyright{}
\date{}

\hypersetup{
 pdfauthor={Yuki Irie},
 pdftitle={Saturations of Mis\`{e}re Nim},
 pdfkeywords={},
 pdfsubject={},
 pdfcreator={Emacs 25.2.1 (Org mode 8.3.4)}, 
 pdflang={English}}
\begin{document}

\maketitle
We consider mis\`{e}re Nim as a normal-play game obtained from Nim by removing the terminal position.
While explicit formulas are known for the Sprague-Grundy functions of Nim and Welter's game,
no explicit formula is known for that of mis\`{e}re Nim.
All three of these games can be considered as position restrictions of Nim.
What are the differences between them?
We point out that Nim and Welter's game are saturated, but mis\`{e}re Nim is not.
Moreover, we present explicit formulas for the Sprague-Grundy functions of saturations 
of mis\`{e}re Nim, which are obtained from mis\`{e}re Nim by adjoining some moves.

\section{Introduction}
\label{sec:orgheadline22}
\label{orgtarget1}

The loser in Nim is the winner in mis\`{e}re Nim.
Nim is a two-player game played with heaps of coins.
Two players alternately choose a heap and take at least one coin from it.
The player who takes the last coin wins in Nim and loses in mis\`{e}re Nim.
In general, the player who moves last wins in the normal-play convention and loses in the mis\`{e}re-play convention,
which has been extensively studied by using mis\`{e}re Sprague-Grundy functions, genera, and mis\`{e}re quotients
(see, for example, \cite{bouton-nim-1901, grundy-disjunctive-1956, conway-numbers-2001, plambeck-taming-2005, plambeck-misere-2008}).
In this paper, we will consider mis\`{e}re Nim as a normal-play game 
obtained from Nim by removing the terminal position (see Section \ref{orgtarget2}).

Impartial games including Nim and mis\`{e}re Nim can be analyzed
using their (normal) Sprague-Grundy functions \cite{sprague-uber-1935, grundy-mathematics-1939}.
Sprague-Grundy functions are defined recursively, and
computing them often leads to a combinatorial explosion.
However, explicit formulas are known for the Sprague-Grundy functions of some games such as Nim \cite{sprague-uber-1935, grundy-mathematics-1939} and Welter's game \cite{welter-theory-1954}, which is a position restriction of Nim.
Though mis\`{e}re Nim is also a position restriction of Nim,
no explicit formula is currently available for its Sprague-Grundy function.\footnote{By contrast, an explicit formula is known for the mis\`{e}re Sprague-Grundy function of Nim. See Remark \ref{orgtarget3}.}

What are the differences between Nim, Welter's game, and mis\`{e}re Nim?
One of the differences is that Nim and Welter's game are \emph{\(2\)-saturated}\footnote{The concept of saturations was first introduced in \cite{irie-psaturations-2018} to connect Welter's game with representations of symmetric groups.}, but mis\`{e}re Nim is not.
The purpose of this paper is to present an explicit formula for the Sprague-Grundy functions of 2-saturations 
of mis\`{e}re Nim, which are obtained from mis\`{e}re Nim by adjoining some moves.
More generally, for a mixed-radix number system \(\ybs\), we give an explicit formula for the Sprague-Grundy functions of \(\ybs\)-saturations of
mis\`{e}re Nim.

\subsection{Mixed-radix number systems}
\label{sec:orgheadline1}
\label{orgtarget4}
We introduce some notation for mixed-radix number systems.

Let \(\NN\) be the set of nonnegative integers.
Throughout this paper, \(\ybs\) denotes a sequence \((\ybs_L)_{L \in \NN} \Tin \NN^\NN\) with \(\ybs[L] \Tge 2\) for every \(L \Tin \NN\).
Define \(\ybsp[L] \Ttobe \ybs[0] \cdot \ybs[1] \cdots \ybs[L - 1]\). 
For example, if \(\ybs \Teq (2,3,2,\ldots)\), then \(\ybsp[0] \Teq 1\), \(\ybsp[1] \Teq 2\), and \(\ybsp[2] \Teq 6\).

Let \(\sgX \Tbein \NN\).
We denote by \(\sgX_L^{\ybs}\) the \(L\)th digit in the mixed base \(\ybs\) expansion of \(\sgX\), that is,
if \(\sgX_{\ge L}^{\ybs}\) is the integer quotient\footnote{\(\sgX_{\ge L}^{\ybs}\) is the unique integer satisfying \(\sgX - \sgX_{\ge L}^{\ybs} \ybsp[L] \in \set{0,1,\ldots,\ybsp[L] - 1}\).} of \(\sgX\) divided by \(\ybsp[L]\), then \(\sgX_L^{\ybs} \Teq \sgX_{\ge L}^{\ybs} \bmod \ybs[L]\),
where \(\sgX_{\ge L}^{\ybs} \bmod \ybs[L]\) is the remainder of \(\sgX_{\ge L}^{\ybs}\) divided by \(\ybs[L]\).
By definition, \[
 \sgX \Teq \sum_{L \in \NN} \sgX_L^{\ybs} \ybsp[L] \quad \tand \quad \sgX^{\ybs}_L \in \set{0,1,\ldots, \ybs[L] - 1}.
\]
For example, if \(\ybs[L] \Teq b\) for every \(L \Tin \NN\), then \(\sgX_L^{\ybs}\) is the \(L\)th digit in the ordinary base \(b\) expansion of \(\sgX\),
so it is convenient to write \(\ybs = b\).
For a negative integer \(\sgX\), we define \(\sgX^{\ybs}_L\) similarly; then
\[
 \sgX^{\ybs}_L + (- \sgX - 1)^{\ybs}_L \Teq \ybs[L] - 1.
\]
For example, \((-1)^{\ybs}_L \Teq \ybs[L] - 1\).
We drop the superscript \(\ybs\) when no confusion can arise. 

For \(\sgX \Tin \ZZ\), define
\[
 \ord_{\ybs}(n) \Ttobe \begin{cases}
 \min \Set{L \in \NN : \sgX_L \neq 0} \text{\ (} = \max \Set{L \in \NN : \ybsp[L] \text{\ divides\ } \sgX } \text{)} & \tif  \sgX \neq 0,\\
 \infty & \tif \sgX = 0.
 \end{cases}
\]
For example, if \(\ybs \Teq (3, 2, 5, 4, \ldots)\), then \(\ord_{\ybs}(54) \Teq \ord_{\ybs}(4 \cdot \ybsp[2] + \ybsp[3]) = 2\).

\subsection{Subtraction games}
\label{sec:orgheadline9}
\label{orgtarget2}
We define subtraction games, their mis\`{e}re versions, and Sprague-Grundy functions.

Fix a positive integer \(\numcoins\) and let \(\Omega\) denote \(\set{0,1,\ldots, \numcoins - 1}\).
Let \(\cP \Tbesubseteq \NN^\numcoins\) and \(\sgC \Tbesubseteq \NN^\numcoins \setminus \set{(0,\ldots,0)}\).
Define \(\Gamma(\cP, \sgC)\) to be the digraph with vertex set \(\cP\) and edge set 
\[
 \set{(X, Y) \in \cP^2 : X - Y \in \sgC}.
\]
We call \(\Gamma(\cP, \sgC)\) a \emph{subtraction game} or a \emph{take-away game}.
The vertex set \(\cP\) is called the \emph{position set} of \(\Gamma(\cP, \sgC)\).

\begin{remark}
 \comment{Rem. how to play}
\label{sec:orgheadline2}
We can consider \(\Gamma(\cP, \sgC)\) as a two-player game as follows.
Before the game begins, we pick an initial position \(X_0 \in \cP\).
The first player subtracts some \(C_0 \in \sgC\) from \(X_0\) so that \(X_0 - C_0 \Tisin \cP\).
Let \(X_1 \Tbe X_0 - C_0\).
Similarly, the second player subtracts some \(C_1 \Tin \sgC\) from \(X_1\) so that \(X_1 - C_1 \Tisin \cP\).
In this way, the two players alternately subtract some \(C \Tin \sgC\) from the current position.
The player who subtracts last wins.
 
\end{remark}

 \begin{example}[Nim]
 \comment{Exm. [Nim]}
\label{sec:orgheadline3}
\label{orgtarget5}
Let
\[
 \sgC_{[1]} \Tbe \set{\ccC \in \NN^\numcoins : \wt(C) = 1},
\]
where \(\wt(C)\) is the Hamming weight of \(C\), that is, the number of nonzero components of \(C\).
The subtraction game \(\Gamma(\NN^\numcoins, \sgC_{[1]})\) is called \emph{Nim}.
For example, in Nim, the first player will win if we start from \((1,0)\);
indeed, he can subtract \((1,0) \Tin \sgC_{[1]}\) from \((1,0)\),
but the second player cannot subtract any \(C \Tin \sgC_{[1]}\) from \((0,0)\).
 
\end{example}

\comment{Misere}
\label{sec:orgheadline4}
\label{orgtarget6}
We next define the mis\`{e}re version of a subtraction game.
Let \(X\) be a position in a subtraction game \(\Gamma(\cP, \sgC)\).
If \(X - C \Tisin \cP\) for some \(C \Tin \sgC\), then \(X - C\) is called an \emph{option} of \(X\) (in \(\Gamma(\cP, \sgC)\)).
If \(X\) has no options, then \(X\) is called a \emph{terminal position}.
Let \(\cP'\) be the set of non-terminal positions in \(\Gamma(\cP, \sgC)\).
The subtraction game \(\Gamma(\cP', \sgC)\) is called the \emph{mis\`{e}re version} of \(\Gamma(\cP, \sgC)\) \cite{kahane-hexad-2001}.

 \begin{example}[mis\`{e}re Nim]
 \comment{Exm. [mis\`{e}re Nim]}
\label{sec:orgheadline5}
Let \(\cPmis \Tbe \cPmis^\numcoins = \NN^\numcoins \setminus \set{(0,\ldots, 0)}\).
Then \(\cPmis\) is the set of non-terminal positions of Nim,
so the mis\`{e}re version of Nim is \(\Gamma(\cPmis, \sgC_{[1]})\).
We call \(\Gamma(\cPmis, \sgC_{[1]})\) \emph{mis\`{e}re Nim}.\footnote{Mis\`{e}re Nim is usually defined to be Nim in mis\`{e}re play, so the definition of mis\`{e}re Nim used in this paper is
slightly different from the standard one. However, for \(X \in \NN^\numcoins \setminus \set{(0,\ldots, 0)}\), the outcomes of \(X\) in the two of these mis\`{e}re Nim are the same.
In other words, when the initial position is \(X\), the first player can win in one of the two mis\`{e}re Nim if and only if he can win in the other mis\`{e}re Nim.}
In mis\`{e}re Nim, 
the first player will lose if we start from \((1,0)\) because this position is terminal.
 
\end{example}

\comment{Sprague-Grundy}
\label{sec:orgheadline6}
We now define Sprague-Grundy functions.
See, for example, \cite{berlekamp-Winning-2001, conway-numbers-2001, albert-lessons-2007, siegel-combinatorial-2013} for details.
Let \(\Gamma \Tbe \Gamma(\cP, \sgC)\) and \(X \Tbein \cP\).
The \emph{Sprague-Grundy value} \(\sg(X)\) of \(X\) is defined to be
the minimum nonnegative integer \(\sgX\) such that \(\sgX\) is not equal to the Sprague-Grundy value of any option of \(X\), that is,
\[
 \sg(X) \Teq \sg_{\Gamma}(X) = \mex \set{\sg_{\Gamma}(Y) : Y \text{\ is an option of\ } X},
\]
where \(\mex S \Teq \min \set{\sgX \in \NN : \sgX \not \in S}\).
Note that if \(X\) is a terminal position, then \(\sg(X) \Teq \mex \emptyset = 0\).
The nonnegative integer-valued function \(\sg: \cP \ni X \mapsto \sg(X) \in \NN\)
is called the \emph{Sprague-Grundy function} of the subtraction game \(\Gamma\).
For a position \(X\),
the following two statements are equivalent:
\begin{enumerate}
\item \(\sg(X) \Teq 0\).
\item The second player can win when the initial position is \(X\).
\end{enumerate}

 \begin{example}
 \comment{Exm.}
\label{sec:orgheadline7}
\label{orgtarget7}
Let \(\numcoins \Tbe 2\).
We calculate the Sprague-Grundy values of some positions in mis\`{e}re Nim (see Table \ref{tab:orgtable3} in Example \ref{orgtarget8}).
Since \((0,1)\) and \((1,0)\) are terminal positions, their Sprague-Grundy values are 0.
Hence \(\sg((0,2)) \Teq \sg((1,1)) = \sg((2,0)) = \mex \set{0} = 1\).
This implies that \(\sg((1,2)) \Teq \sg((2,1)) = \mex \set{0,1} = 2\),
so \(\sg((2,2)) \Teq \mex \set{1,2} = 0\).
We can verify that the second player can win when \((2,2)\) is the initial position.
 
\end{example}

\begin{remark}
 \comment{Rem.}
\label{sec:orgheadline8}
\label{orgtarget3}
The Sprague-Grundy function of mis\`{e}re Nim is different from
the mis\`{e}re Sprague-Grundy function \(\mathscr{G}^{-}\) of Nim defined in \cite{conway-numbers-2001}.
The domain of the former function is \(\NN^\numcoins \setminus \set{(0,\ldots, 0)}\) and that of the latter one is \(\NN^\numcoins\).
Here, for \(X \in \NN^\numcoins\), we can compute \(\mathscr{G}^-(X)\) as follows:
\[
 \mathscr{G}^{-}(X) \Teq \begin{cases}
 1 & \tif X = (0, \ldots, 0), \\
 \mex \set{\mathscr{G}^-(Y) : Y \text{\ is an option of\ } X} & \totherwise.
 \end{cases}
\]
The value \(\mathscr{G}^{-}(X)\) is generally not equal to \(\sg_{\Gamma}(X)\),
where \(\Gamma\) is mis\`{e}re Nim \(\Gamma(\cPmis, \sgC_{[1]})\).
For example, \(\mathscr{G}^{-}((0,2)) \Teq 2 \neq 1 = \sg_{\Gamma}((0,2))\).
However, \(\mathscr{G}^{-}(X) \Teq 0\) if and only if \(\sg_{\Gamma}(X) \Teq 0\).\footnote{As we have mentioned, \(\mathscr{G}^{-}(X)\) can be written down explicitly.
If \(\max X \Tisge 2\), then \(\mathscr{G}^{-}(X) \Teq \nsigma[2](X)\), where \(\nsigma[2](X)\) is the Nim sum of the components of \(X\).
If \(\max X \Tislt 2\), then \(\mathscr{G}^{-}(X) \Teq 1 - \nsigma[2](X)\).}
 
\end{remark}

\subsection{\texorpdfstring{\(\ybs\)}{beta}-Saturations}
\label{sec:orgheadline13}
\label{orgtarget9} 

\comment{Def 2-sat}
\label{sec:orgheadline10}
We define \(\ybs\)-saturations of subtraction games.

Elements in \(\NN^\numcoins\) will be denoted by upper-case letters, and components of them by lower-case letters with superscripts.
For example, \(C = (c^0, \ldots, c^{\numcoins - 1}) \in \NN^\numcoins\).
Define
\[
 \sgCord[\ybs] = \sgCord[\ybs,\numcoins] = \Set{\ccC \in \NN^\numcoins \setminus \set{(0,\ldots,0)} : \ord_{\ybs}\left(\sum_{i \in \Omega} \ccc^i\right) = \mord_{\ybs}(C)},
\]
where
\[
 \mord_{\ybs}(C) \Teq \min \set{\ord_{\ybs}(c^i) : i \in \Omega}.
\]
For example, \((2,2,6) \Tisin \sgCord[2]\) and \((2,2,4) \not \in \sgCord[2]\) because
\[
 \ord_2(2 + 2 + 6) \Teq 1 = \mord_2((2,2,6)) \ \tand \ \ord_2(2 + 2 + 4) = 3 > 1 =  \mord_2((2,2,4)).
\]
A subtraction game \(\Gamma(\cP, \sgC)\) is said to be \emph{\(\ybs\)-saturated} if its Sprague-Grundy function
is equal to that of \(\Gamma(\cP, \sgCord[\ybs])\). If \(\Gamma(\cP, \sgC)\) is \(\ybs\)-saturated,
then we also say that it is a \emph{\(\ybs\)-saturation} of \(\Gamma(\cP, \sgC_{[1]})\).

 \begin{example}[Nim and Welter's game]
 \comment{Exm. [Nim and Welter's game]}
\label{sec:orgheadline11}
\label{orgtarget10}
Let
\[
 \cP_{\text{Wel}} \Tbe \set{X \in \NN^\numcoins : x^i \neq x^j \ttext{whenever} i \neq j}.
\]
The subtraction game \(\Gamma(\cP_{\text{Wel}}, \sgC_{[1]})\) is called \emph{Welter's game}.
It is known that Nim and Welter's game are 2-saturated \cite{fraenkel-Nimhoff-1991, blass-how-1998, irie-psaturations-2018, welter-theory-1954},
that is, for \(\cP \Tin \set{\NN^\numcoins, \cP_{\text{Wel}}}\), the Sprague-Grundy function of \(\Gamma(\cP, \sgC_{[1]})\) is equal to that of \(\Gamma(\cP, \sgCord[2])\).
Moreover, \(\Gamma(\NN^\numcoins, \sgC)\) is 2-saturated if and only if \(\sgC_{[1]} \subseteq \sgC \subseteq \sgCord[2]\) \cite{blass-how-1998}.
 
\end{example}

 \begin{example}
 \comment{Exm.}
\label{sec:orgheadline12}
\label{orgtarget8}
Let \(\numcoins \Tbe 2\).
We compare the Sprague-Grundy function of mis\`{e}re Nim \(\Gamma(\cPmis, \sgC_{[1]})\) with that of \(\Gamma(\cPmis, \sgCord[2])\) (see Table \ref{tab:orgtable3}).
We first consider the position \((2,2)\).
The Sprague-Grundy value of \((2,2)\) equals 0 in \(\Gamma(\cPmis, \sgC_{[1]})\); 
however, it equals 3 in \(\Gamma(\cPmis, \sgCord[2])\) because \((0, 1)\) is an option of \((2, 2)\) in \(\Gamma(\cPmis, \sgCord[2])\).
Thus mis\`{e}re Nim is not 2-saturated when \(\numcoins = 2\).
We next compute the Sprague-Grundy value of \((2, 3)\) in \(\Gamma(\cPmis, \sgCord[2])\).
Since \((2, 3) - (0, 1) \Teq (2, 2) \not \in \sgCord[2]\) and \((2, 3) - (1, 0) \Teq (1, 3) \not \in \sgCord[2]\),
\((2, 3)\) has no options with Sprague-Grundy value 0, and hence its Sprague-Grundy value is 0.

\begin{table}[H]
\caption{\label{tab:orgtable3}
Sprague-Grundy values in $\Gamma(\cPmis, \sgC_{[1]})$ and $\Gamma(\cPmis, \sgCord[2])$.}
\begin{tabular}{cc}
\begin{minipage}{0.49\hsize}
\begin{table}[H]
\centering
\begin{tabular}{c | p{0.4em} p{0.4em} p{0.4em} p{0.4em} p{0.4em} p{0.4em} p{0.4em} p{0.4em} p{0.4em}}
   & 0 & 1 & 2 & 3 & 4 & 5 & 6 & 7 & 8 \\ 
\hline
 0 &   & 0 & 1 & 2 & 3 & 4 & 5 & 6 & 7 \\  
 1 & 0 & 1 & 2 & 3 & 4 & 5 & 6 & 7 & 8 \\  
 2 & 1 & 2 & 0 & 4 & 5 & 3 & 7 & 8 & 6 \\  
 3 & 2 & 3 & 4 & 0 & 1 & 6 & 8 & 5 & 9 \\  
 4 & 3 & 4 & 5 & 1 & 0 & 2 & 9 & 10 & 11 \\  
 5 & 4 & 5 & 3 & 6 & 2 & 0 & 1 & 9 & 10 \\  
 6 & 5 & 6 & 7 & 8 & 9 & 1 & 0 & 2 & 3 \\  
7 & 6 & 7 & 8 & 5 & 10 & 9 & 2 & 0 & 1 \\  
 8 & 7 & 8 & 6 & 9 & 11 & 10 & 3 & 1 & 0 \\  
\end{tabular}
\end{table}
\end{minipage}
\begin{minipage}{0.49\hsize}
\begin{table}[H]
\centering
\begin{tabular}{c | p{0.4em} p{0.4em} p{0.4em} p{0.4em} p{0.4em} p{0.4em} p{0.4em} p{0.4em} p{0.4em}}
  &  0 & 1 & 2 & 3 & 4 & 5 & 6 & 7 & 8 \\ 
\hline
 0 &   & 0  & 1  & 2  & 3  & 4  & 5  & 6  & 7    \\ 
 1 & 0  & 1  & 2  & 3  & 4  & 5  & 6  & 7  & 8    \\ 
 2 & 1  & 2  & 3  & 0  & 5  & 6  & 7  & 4  & 9    \\ 
 3 & 2  & 3  & 0  & 1  & 6  & 7  & 4  & 5  & 10    \\ 
 4 & 3  & 4  & 5  & 6  & 7  & 0  & 1  & 2  & 11    \\ 
 5 & 4  & 5  & 6  & 7  & 0  & 1  & 2  & 3  & 12    \\ 
 6 & 5  & 6  & 7  & 4  & 1  & 2  & 3  & 0  & 13    \\ 
 7 & 6  & 7  & 4  & 5  & 2  & 3  & 0  & 1  & 14    \\ 
 8 & 7  & 8  & 9  & 10  & 11  & 12  & 13  & 14  & 15    \\ 
\end{tabular}
\end{table}
\end{minipage}
\end{tabular}
\end{table}
 
\end{example}

\subsection{A formula for \texorpdfstring{\(\ybs\)}{beta}-saturations of mis\`{e}re Nim}
\label{sec:orgheadline20}
\label{orgtarget11}
We present an explicit formula for the Sprague-Grundy functions of \(\ybs\)-saturations of mis\`{e}re Nim.

Let \(\range{\ybs[L]}\) denote \(\set{0,1,\ldots, \ybs[L] - 1}\) equipped with the following two operations:
for \(a, b \Tin \range{\ybs[L]}\),
\[
 a \oplus b \Teq (a + b) \bmod \ybs[L]  \quad \tand \quad a \ominus b = (a - b) \bmod \ybs[L].
\]
In other words, \(\range{\ybs[L]}\) is the additive group of integers modulo \(\ybs[L]\).
For example, in \(\range{5}\), \(2 \oplus 4 \Teq 1\) and \(2 \ominus 4 \Teq 3\).
For \(\sgX \Tin \NN\), we will think of \(\sgX_L^{\ybs}\) as an element of \(\range{\ybs[L]}\).
Let \(\Zadic\) denote \(\prod_{L \in \NN} \range{\ybs[L]}\) equipped with the following two operations: 
for \(\sgX, \sgY \Tin \Zadic\),
\[
 \sgX \oplus \sgY \Teq \bnums{\sgX_L \oplus \sgY_L}_{L \in \NN} \quad \tand \quad \sgX \ominus \sgY = \bnums{\sgX_L \ominus \sgY_L}_{L \in \NN},
\]
where \(\sgX = \bnums{\sgX_L}_{L \in \NN}\) and \(\sgY = \bnums{\sgY_L}_{L \in \NN}\).
For \(\sgX \in \Zadic\), define
\[
 \ord_{\ybs}(\sgX) \Ttobe \begin{cases}
 \min \Set{L \in \NN : \sgX_L \neq 0}& \tif  \sgX \neq \bnums{0,0,\ldots},\\
 \infty & \tif \sgX = \bnums{0,0,\ldots}.
 \end{cases}
\]
Consider the map \(\Phi : \ZZ \ni \sgX \mapsto \bnums{\sgX^{\ybs}_L}_{L \in \NN} \in \Zadic\).
For \(n \in \ZZ\), we identify \(\Phi(n)\) with \(n\).
Let \(\Nb\) denote \(\Phi(\NN)\).
For \(n \in \Nb\), it is convenient to write \(\sgX = \bnums{\sgX_0, \sgX_1, \ldots, \sgX_{L - 1}}\) when \(\sgX_{\ge L} \Teq 0\). 
For example, if \(\ybs = 10\), then \(24 = \bnums{4, 2, 0, \ldots} = \bnums{4, 2} \in \Nb[10]\).

For \(X \Tin \Nb^\numcoins\), define
\begin{equation}
\label{orgtarget12}
 \nsigma[\ybs](X) \Ttobe \nsigma[\ybs, \numcoins](X) = x^0 \oplus \cdots \oplus x^{\numcoins - 1} \text{\ (}\in \Nb \text{)}.
\end{equation}
Let \(\nsigma[\ybs]<L>(X) \Tbe (\nsigma[\ybs](X))_L\) (\(=\) the \(L\)th digit of \(\nsigma[\ybs](X)\)) for \(L \in \NN\).

 \begin{example}[\(\ybs\)-saturations of Nim \cite{irie-mixedradix-2018}]
 \comment{Exm. [\(\ybs\)-saturations of Nim \cite{irie-mixedradix-2018}]}
\label{sec:orgheadline14}
\label{orgtarget13}
Let \(X\) be a position in a \(\ybs\)-saturation of Nim.
Then 
\[
 \sg(X) \Teq \nsigma[\ybs](X).
\]
For example, if \(\ybs \Teq (3,2,5,\ldots)\) and \(X \Teq (16, 27) = (\bnums{1, 1, 2}, \bnums{0, 1, 4}) \in \Nb^2\), then
\[
 \sg(X) \Teq \nsigma[\ybs](X) = \bnums{1 \oplus 0, 1 \oplus 1, 2 \oplus 4} = \bnums{1, 0, 1} = 7.
\]
 
\end{example}

 \begin{example}[\(b\)-saturations of Welter's game \cite{irie-psaturations-2018}]
 \comment{Exm. [\(b\)-saturations of Welter's game \cite{irie-psaturations-2018}]}
\label{sec:orgheadline15}
\label{orgtarget14}
Let \(X\) be a position in a \(b\)-saturation of Welter's game,
where \(b\) is an integer greater than 1. Then
\[
 \sg(X) \Teq \nsigma[b](X) \oplus \bigoplus_{i < j} \left(b^{\ord_b(x^i - x^j) + 1} - 1 \right).
\]
For example, if \(b = 3\) and \(X = (1, 4) \in \Nb[3]^2\), then \(\sg(X) = 1 \oplus 4 \oplus (3^2 - 1) = 1\).
 
\end{example}
\comment{misere}
\label{sec:orgheadline16}
We now give an explicit formula for the Sprague-Grundy functions of \(\ybs\)-saturations of mis\`{e}re Nim.
For \(X \Tin \cPmis\) (\(\subseteq \Nb^\numcoins\)), define
\begin{equation}
\label{orgtarget15}
\mphi[\ybs](X) \Ttobe \nsigma[\ybs](X) \oplus \left(\ybsp[\mord_{\ybs}(X) + 1] - 1 \right).
\end{equation}

 \begin{theorem}
 \comment{Thm. 2-misere}
\label{sec:orgheadline17}
\label{orgtarget16}
The Sprague-Grundy function of a \(\ybs\)-saturation of mis\`{e}re Nim
is equal to \(\mphi[\ybs]\), that is,
\[
 \sg(X) \Teq \mphi[\ybs](X)
\]
for every position \(X\) in a \(\ybs\)-saturation of mis\`{e}re Nim.
 
\end{theorem}

\comment{connect}
\label{sec:orgheadline18}
Before giving an example of Theorem \ref{orgtarget16}, we introduce some notation. 
For \(\sgX \Tin \Nb\) and \(L \Tin \NN\), let
\(\sgX_{<L} \Tbe [\sgX_0, \sgX_1, \ldots, \sgX_{L - 1}] \in \Nb\), that is, \(\sgX_{<L} \Teq \sgX \bmod \ybsp[L]\).
For \(X \Tin \Nb^\numcoins\), let \(X_{<L} \Tbe (x^0_{<L}, x^1_{<L} \ldots, x^{\numcoins - 1}_{<L}) \in \Nb^\numcoins\).
When \(X \Teq X_{<L}\),
it is convenient to write \(X\) as follows:
\[
 X = \begin{bmatrix}
 x^0_0 & x^0_1 & \cdots & x^0_{L - 1}  \\
 \vdots & \vdots & \ddots & \vdots \\
 x^{\numcoins - 1}_0 & x^{\numcoins - 1}_1 & \cdots & x^{\numcoins - 1}_{L - 1} \\
 \end{bmatrix}.
\]
Note that if \(\mord_{\ybs}(X) \Teq \cMX\), then
\[
 X \Teq \begin{bmatrix}
  &  &   x^0_{\cMX} & \cdots & x^0_{L - 1}\\
  & \bigzero  &  \vdots & \cdots & \vdots \\
  &  &   x^{\numcoins - 1}_{\cMX} & \cdots & x^{\numcoins - 1}_{L - 1}
 \end{bmatrix}.
\]

 \begin{example}
 \comment{Exm. 2-misere:2}
\label{sec:orgheadline19}
\label{orgtarget17}
Let us consider Example \ref{orgtarget8} again.
Let
\[
 X \Tbe (2,2) = \begin{bmatrix}
 0 & 1 \\
 0 & 1
 \end{bmatrix} \in \Nb[2]^2 \quad \tand \quad
 Y = (2,3) = \begin{bmatrix}
 0 & 1 \\
 1 & 1
 \end{bmatrix} \in \Nb[2]^2.
\]
Then \(\mord_{2}(X) \Teq 1\) and \(\mord_{2}(Y) \Teq 0\), so
\[
 \mphi[2](X) \Teq \nsigma[2](X) \oplus (2^{1 + 1} - 1) = 3
\quad \tand \quad
 \mphi[2](Y) = \nsigma[2](Y) \oplus (2^{0 + 1} - 1) = 0.
\]
 
\end{example}

\subsection{The weight of \texorpdfstring{$\mphi[\ybs, \numcoins]$}{phi(beta, k)}}
\label{sec:orgheadline21}
\label{orgtarget18}
We give the minimum of the \emph{weight} of \(\sgC\) such that
\(\Gamma(\cPmis, \sgC)\) is a \(\ybs\)-saturation of mis\`{e}re Nim.

Let \(\cP \Tbesubseteq \NN^\numcoins\).
For a nonnegative integer-valued function \(\psi : \cP \to \NN\),
let \(\Delta(\psi)\) be the set of \(\sgC \subseteq \NN^\numcoins \setminus \set{(0,\ldots,0)}\) 
such that the Sprague-Grundy function of \(\Gamma(\cP, \sgC)\) equals \(\psi\).
Note that if \(\sgC, \sgD \Tisin \Delta(\psi)\) and \(\sgC \Tissubseteq \sgE \subseteq \sgD\), then \(\sgE \Tisin \Delta(\psi)\).
By definition, \(\Gamma(\cP, \sgC)\) is \(\ybs\)-saturated if and only if
\(\sgC \Tisin \Delta(\psi^{\ybs})\), where \(\psi^{\ybs}\) is the Sprague-Grundy function of \(\Gamma(\cP, \sgCord[\ybs])\).
If \(\Delta(\psi) \Tneq \emptyset\), then define
\[
 \wt(\psi) \Ttobe \min_{\sgC \in \Delta(\psi)} \wt(\sgC), 
\]
where \(\wt(\sgC) \Teq \max \set{\wt(C) : C \in \sgC}\) and \(\max \emptyset = 0\).
For example, if \(\psi^2\) is the Sprague-Grundy function of a 2-saturation of Nim or that of a 2-saturation of Welter's game,
then 
\[
 \wt(\psi^2) \Teq 1
\] 
since \(\sgC_{[1]} \Tisin \Delta(\psi^2)\). In other words, Nim and Welter's game themselves are 2-saturated.
However, as we have seen in Example \ref{orgtarget8}, if \(\numcoins \Teq 2\), then \(\sgC_{[1]} \Tisnotin \Delta(\mphi[2,2])\), 
so \(\wt(\mphi[2,2]) \Teq 2\). In fact, if \(\numcoins \Tisge 2\), then
\[
 \wt(\mphi[2,\numcoins]) \Teq 2.
\]
Let \(\ccB\) be the supremum of \(\set{\ybs[L] : L \ge 1}\) in \(\NN \cup \set{\infty}\). 
In general, we will prove that
\begin{eqnarray}
\label{orgtarget19}
 \wt(\mphi[\ybs, \numcoins]) &\Teq& \max \bigg\{\min \Big\{\ybs[L] - \delta(L), \numcoins - \delta(L)[\ybs[0] < 2\numcoins] \Big\} : L \in \NN \bigg\} \\
 &=& \begin{cases}
 \numcoins & \tif \ccB \ge \numcoins \ \  \tor \ \ \ \ybs[0] \ge 2\numcoins, \\
 \numcoins - 1 & \tif \ccB < \numcoins \ \tand \ \numcoins \le \ybs[0] < 2\numcoins, \\
 \max \set{\ybs[0] - 1, \ccB}  & \tif \ccB < \numcoins \ \tand \ \ybs[0] < \numcoins,
 \end{cases} \nonumber
\end{eqnarray}
where \(\delta(L) \Teq [L = 0]\) and \([\ ]\) is the Iverson bracket notation, that is, \([P] \Teq 1\) if a statement \(P\) holds, and \([P] = 0\) otherwise.
In particular, if \(\ybs \Teq b\) for some \(b \Tin \NN\), then
\[
 \wt(\mphi[b, \numcoins]) \Teq \min \set{b, \numcoins}.
\]

\section{Proofs}
\label{sec:orgheadline53}
When no confusion can arise,
we write \(\nsigma\) and \(\mphi\) instead of
\(\nsigma[\ybs]\) and \(\mphi[\ybs]\), respectively.

\subsection{Preliminaries}
\label{sec:orgheadline23}
\label{orgtarget20}
Let \(\sgC \Tbesubseteq \Nb^\numcoins \setminus \set{(0, \ldots, 0)}\).
The Sprague-Grundy function of \(\Gamma(\cPmis, \sgC)\) equals \(\mphi\) if and only if \(\sgC\) satisfies the following two conditions:
\begin{description}
\item[{(SG1)}] \label{orgtarget21} If \(X \Tisin \cPmis\), then \(X\) has no option \(Y\) with \(\mphi(Y) = \mphi(X)\) in \(\Gamma(\cPmis, \sgC)\), that is, \(\mphi(X - C) \Tneq \mphi(X)\) for every \(C \Tin \sgC\) with \(X - C \Tin \cPmis\).
\end{description}
\phantomsection
\begin{description}
\item[{(SG2)}] \label{orgtarget22} If \(X \Tisin \cPmis\) and \(0 \Tisle \sgY < \mphi(X)\), then \(X\) has an option \(Y\) with \(\mphi(Y) =  \sgY\) in \(\Gamma(\cPmis, \sgC)\), that is, \(\mphi(X - C) \Teq \sgY\) for some \(C \Tin \sgC\) with \(X - C \Tin \cPmis\).
\end{description}

To prove Theorem \ref{orgtarget16} and (\ref{orgtarget19}), it therefore suffices to show the following three assertions:
\begin{description}
\item[{(A1)}] \(\sgCord[\ybs]\) satisfies (\hyperref[orgtarget21]{SG1}).
\item[{(A2)}] \(\set{C \in \Nb^\numcoins : \wt(C) < w}\) does not satisfy (\hyperref[orgtarget22]{SG2}), where \(w\) is the right-hand side of (\ref{orgtarget19}).
\item[{(A3)}] \(\sgCord[\ybs][[w]]\) satisfies (\hyperref[orgtarget22]{SG2}), where \(\sgCord[\ybs][[w]] \Teq \set{C \in \sgCord[\ybs] : \wt(C) \le w}\).
\end{description}

\subsection{Proof of (A1)}
\label{sec:orgheadline28}
\label{orgtarget23}

\comment{Goal}
\label{sec:orgheadline24}

Let \(X \Tbein \cPmis\), \(C \Tbein \sgCord[\ybs]\) with \(X - C \Tin \cPmis\), and \(Y = X - C\).
Let \(\cMX = \mord_{\ybs}(X), \cMY = \mord_{\ybs}(Y)\), and \(\cMC = \mord_{\ybs}(C)\).
We show that \(\mphi<\cMC>(Y) \Tneq \mphi<\cMC>(X)\),
where \(\mphi<\cMC>(X) \Teq (\mphi(X))_\cMC\).
Since \((\ybs^{\cMX + 1} - 1)_L \Teq \ominus \Compare{L}{\cMX}\) for \(L \Tin \NN\), it follows that
\begin{equation}
\label{orgtarget24}
 \mphi<L>(X) \Teq \nsigma<L>(X) \ominus \Compare{L}{\cMX} \quad \tand \quad \mphi<L>(Y) = \nsigma<L>(Y) \ominus \Compare{L}{\cMY}.
\end{equation}

\comment{[L \(\le\) M(X)] = [L \(\le\) M(Y)]}
\label{sec:orgheadline25}
We first show that
\begin{equation}
\label{orgtarget25}
\Compare{\cMC}{\cMX} \Teq \Compare{\cMC}{\cMY}.
\end{equation}
Since \(C_{<\cMC} \Teq (0,\ldots, 0)\), we see that \(Y_{<\cMC} \Teq (X - C)_{< \cMC} = X_{<\cMC}\).
Suppose that \(\cMC \Tisgt \cMX\). Then \(Y_{\le \cMX} \Teq X_{\le \cMX} \neq (0, \ldots, 0)\), where \(Y_{\le \cMX} \Teq Y_{< \cMX + 1}\).
Hence \(\cMY \Teq \cMX\), so (\ref{orgtarget25}) holds.
If \(\cMC \Tisle \cMX\), then \(Y_{<\cMC} \Teq X_{<\cMC} = (0, \ldots, 0)\),  so \(\cMC \Tisle \cMY\). Therefore (\ref{orgtarget25}) holds.

\comment{\nsigma<\(cMC\)>(X) \(\neq\)\nsigma<\(cMC\)>(Y)}
\label{sec:orgheadline26}
We next show that
\begin{equation}
\label{orgtarget26}
\nsigma<\cMC>(X) \Tneq \nsigma<\cMC>(Y).
\end{equation}
Since \(C_{<\cMC} \Teq (0, \ldots, 0)\), it follows that \(y^i_\cMC \Teq x^i_\cMC \ominus c^i_\cMC\) for \(i \in \Omega\). Hence
\[
 \nsigma<\cMC>(Y) \Teq \nsigma<\cMC>(X) \ominus \nsigma<\cMC>(C).
\]
Recall that \(\cMC \Teq \mord_{\ybs}(C) = \ord_{\ybs} (\sum_i c^i)\) since \(C \Tisin \sgCord[\ybs]\).
This implies that \((\sum_i c^i)_\cMC \Teq \nsigma<\cMC>(C) \neq 0\). Thus (\ref{orgtarget26}) holds.

\comment{Conclusion}
\label{sec:orgheadline27}
Combining (\ref{orgtarget24})--(\ref{orgtarget26}), we see that \(\mphi<\cMC>(X) \Tneq \mphi<\cMC>(Y)\).
Therefore \(\sgCord[\ybs]\) satisfies (\hyperref[orgtarget21]{SG1}).
\qed

\subsection{Proof of (A2)}
\label{sec:orgheadline33}
\label{orgtarget27}
\comment{Goal}
\label{sec:orgheadline31}
A position \(Y \Tin \cPmis\) is called a \emph{descendant} of a position \(X \Tin \cPmis\) if \(X - Y \Tisin \Nb^\numcoins\).

If \(\numcoins \Teq 1\), then \(w \Teq 1\), so (A2) is obvious.
Suppose that \(\numcoins \Tisge 2\).
It suffices to show that there exist \(X \Tin \cPmis\) and \(\sgY\) with \(0 \le \sgY < \mphi(X)\) satisfying the following condition:
if \(Y\) is a descendant of \(X\) with \(\mphi(Y) = \sgY\), then \(\wt(X - Y) \Tisge w\).
By (\ref{orgtarget19}),
\[
 w  \Teq \min \Big\{\ybs[\cMX] - \delta(\cMX), \numcoins - \delta(\cMX)[\ybs[0] < 2\numcoins]\Big\} \tforsome \cMX \in \NN.
\]
Note that \(w \Tisge 1\). We divide the proof into two cases.

\resetmycase
 
 \begin{mycase}[\(\cMX > 0\) or $\ybs_0 < 2\numcoins$]
 \comment{Case. [\(\cMX > 0\) or $\ybs_0 < 2\numcoins$]}
\label{sec:orgheadline29}
We see that \(w  \Teq \min \set{\ybs[\cMX], \numcoins} - \delta(\cMX)\).
Let
\[
 X \Tbe (\underbrace{\ybsp[\cMX], \ldots, \ybsp[\cMX]}_{w + \delta(\cMX)}, 0,\ldots, 0) \in \cPmis^\numcoins.
\]
Since \(w + \delta(\cMX) \Tisle \ybs[\cMX]\) and \(\mord_{\ybs}(X) \Teq \cMX\), it follows that
\begin{align*}
 \mphi(X) &\Teq \nsigma(X) \oplus (\ybsp[\cMX + 1] - 1) = [\overbrace{0,\ldots, 0}^{\cMX}, \underbrace{1 \oplus \cdots \oplus 1}_{w + \delta(\cMX)}] \ominus [\overbrace{1, \ldots, 1}^\cMX, 1] \\
 &=  \begin{cases}
  w \ybsp[\cMX] -  1 & \tif \cMX > 0, \\
 w & \tif \cMX = 0.
 \end{cases}
\end{align*}
In particular, \(\mphi(X) \Tisgt 0\).
Let \(Y\) be a descendant of \(X\) with \(\mphi(Y) = 0\).
To prove that \(\wt(X - Y) \Teq w\),
we show that \(\sum_i y^i_{\cMX} \Teq \delta(\cMX)\).
Since \(\mphi<0>(Y) \Teq 0\), we see that \(\mord_{\ybs}(Y) \Teq 0\).
Hence \(\mphi<\cMX>(Y) \Teq \nsigma<\cMX>(Y) \ominus \Compare{\cMX}{0} = \nsigma<\cMX>(Y) \ominus \delta(\cMX)\). 
Since \(\mphi<\cMX>(Y) \Teq 0\),
\begin{equation}
\label{orgtarget28}
 \nsigma<\cMX>(Y) \Teq \delta(\cMX).
\end{equation}
We also see that \(\sum_i y^i_{\cMX} \Tislt \sum_i x^i_\cMX = w + \delta(\cMX) \le \ybs[\cMX]\) because \(Y\) is a descendant of \(X\) with \(Y \neq X\).
Hence \(\sum_i y^i_{\cMX} \Teq \nsigma<\cMX>(Y) = \delta(\cMX)\).
Therefore \(\wt(X - Y) \Teq w\).
 
\end{mycase}

 \begin{mycase}[\(\cMX = 0\) and $\ybs_0 \ge 2\numcoins$]
 \comment{Case. [\(\cMX = 0\) and $\ybs_0 \ge 2\numcoins$]}
\label{sec:orgheadline30}
Since \(\numcoins \ge 2\),  we see that \(\ybs[0] \Tisge 4\) and \(w \Teq \min \set{\ybs[0] - 1, \numcoins} = \numcoins\).
Let
\[
 X \Tbe (2, \ldots, 2) \in \cPmis^\numcoins.
\]
Then \(\mphi(X) \Teq \mphi<0>(X) = 2\numcoins - 1 > 0\).
Let \(Y\) be a descendant of \(X\) with \(\mphi(Y) = 0\).
Then \(\mphi<0>(Y) \Teq \nsigma<0>(Y) \ominus 1 = 0\).
Since \(\sum_i y_0^i \Tislt \sum_i x_0^i = 2\numcoins \le \ybs[0]\), it follows that \(\sum_i y_0^i \Teq \nsigma<0>(Y) = 1\). 
This implies that \(y^i \Tisin \set{0, 1}\) for \(i \Tin \Omega\). Hence \(\wt(X - Y) \Teq \numcoins = w\).
\qed
 
\end{mycase} 

 \begin{example}
 \comment{Exm. A2}
\label{sec:orgheadline32}
\label{orgtarget29}
Let \(\ybs \Tbe (6, 2, 2, \ldots)\), \(\numcoins \Tbe 3\), and \(X \Tbe (2,2,2)\).
Then \(\mphi(X) \Teq 5\).
If \(Y\) is a descendant of \(X\) with \(\mphi(Y) = 0\), then
\(Y \Tisin \set{(0,0,1), (0,1,0), (1,0,0)}\), and hence \(\wt(X - Y) \Teq 3\).
Note that if \(\ybs \Teq (5,2,2, \ldots)\), then \(\mphi((2,2,2)) \Teq 0\).
 
\end{example}

\subsection{Proof of (A3)}
\label{sec:orgheadline52}
\label{orgtarget30}
To prove (A3), we present two lemmas.

\comment{Solution formula}
\label{sec:orgheadline34}
For \(X \Tin \cPmis\), the next lemma allows us to express \(x^0\) with \(x^1, \ldots, x^{\numcoins - 1}\), and \(\mphi(X)\).

 \begin{lemma}
 \comment{Lem. solution formula}
\label{sec:orgheadline35}
\label{orgtarget31}
Let \(X \Tbein \cPmis\) and \(\sgY \Tbein \Nb\).
For \(i \Tin \Omega\), let
\[
 \cMY^{(i)} \Tbe \mord_{\ybs}\left((\sgY \ominus (-1), x^0, \ldots, x^{i - 1}, x^{i + 1}, \ldots, x^{\numcoins - 1})\right),
\]
\[
 y^{(i)} \Tbe \sgY \ominus (\ybsp[\cMY^{(i)} + 1] - 1) \ominus x^0 \ominus \cdots \ominus x^{i - 1} \ominus x^{i + 1} \ominus \cdots \ominus x^{\numcoins - 1},
\]
and
\[
 Y^{(i)} \Tbe (x^0, \ldots, x^{i - 1}, y^{(i)}, x^{i + 1}, \ldots, x^{\numcoins - 1}).
\]
Then \(\mord_{\ybs}(Y^{(i)}) \Teq \cMY^{(i)}\). In particular, \(\phi(Y^{(i)}) \Teq \sgY\).
 
\end{lemma}

\begin{proof}
 \comment{Proof.}
\label{sec:orgheadline36}
It suffices to prove the lemma when \(i \Teq 0\).
Let \(\cMY \Tbe \cMY^{(0)}\) and \(Y \Tbe Y^{(0)}\). We show that \(\mord_{\ybs}(Y) \Teq \cMY\).
For \(L \Tlt \cMY\),
\begin{align*}
 y^0_L  &\Teq \sgY_L \ominus (\ybsp[\cMY + 1] - 1)_L \ominus x^1_L \ominus \cdots \ominus x^{\numcoins - 1}_L \\
 &= \sgY_L \ominus (\ominus 1) = (\sgY \ominus (-1))_L = 0.
\end{align*}
This implies that \(\mord_{\ybs}(Y) \Tisge \cMY\).
By the definition of \(\cMY\), we see that \((\sgY \ominus (-1))_{\cMY} \Tneq 0\) or \(x^\ccj_{\cMY} \neq 0\) for some \(\ccj \ge 1\).
If the latter holds, then \(\mord_{\ybs}(Y) \Teq \cMY\). Suppose that \((\sgY \ominus (-1))_{\cMY} \Tneq 0\) and \(x^\ccj_{\cMY} \Teq 0\) for every \(\ccj \ge 1\).
Then \(y^0_{\cMY} \Teq \sgY_\cMY \ominus (\ybsp[\cMY + 1] - 1)_{\cMY} \ominus x^1_\cMY \ominus \cdots \ominus x^{\numcoins - 1}_\cMY = (\sgY \ominus (-1))_{\cMY} \neq 0\), so \(\mord_{\ybs}(Y) = \cMY\). 
Therefore 
\[
 \sgY \Teq y^0 \oplus x^1 \oplus \cdots \oplus x^{\numcoins - 1} \oplus (\ybsp[\cMY + 1] - 1) = \nsigma[\ybs](Y) \oplus (\ybsp[\mord_{\ybs}(Y) + 1] - 1) = \mphi(Y).
\]
\end{proof}

 \begin{example}
 \comment{Exm.}
\label{sec:orgheadline37}
\label{orgtarget32}
Let \(\ybs \Tbe 3\), \(X \Tbe (3, 4)\), and \(\sgY \Tbe 2\).
Note that \(\sgY \ominus (-1) \Teq [0,1,1,\ldots]\) and \(\ord_3(\sgY \ominus (-1)) \Teq 1\).
Since
\[
 \cMY^{(0)} \Teq \mord_{3}\big((\sgY \ominus (-1), 4)\big) = 0 \tand \cMY^{(1)} = \mord_{3}\big((\sgY \ominus (-1), 3)\big) = 1,
\]
it follows that
\[
 y^{(0)} \Teq \sgY \ominus (3^1 - 1) \ominus 4 = 8 \tand y^{(1)} = \sgY \ominus (3^2 - 1) \ominus 3 = 0.
\]
By Lemma \ref{orgtarget31}, \(\mphi((y^{(0)}, x^1)) \Teq \mphi((x^0, y^{(1)})) = \sgY = 2\). Indeed,
\[
 \mphi((8,4)) \Teq 8 \oplus 4 \oplus (3^1 - 1) = 2 \tand \mphi((3,0)) = 3 \oplus 0 \oplus (3^2 - 1) = 2.
\]
 
\end{example}

\comment{Connect}
\label{sec:orgheadline38}
The following trivial lemma will be used to construct appropriate options.
For \(X \Tin \Nb^\numcoins\) and \(L \Tin \NN\), let \(x_L \Tbe (x^0_L, x^1_L, \ldots, x^{\numcoins - 1}_L) \in \range{\ybs[L]}^\numcoins\).
For example, if \(\ybs = 3\) and \(X = (\bnums{1,0,2}, \bnums{2,1})\), then \(x_0 = (1,2)\) and \(x_1 = (0, 1)\).

 \begin{lemma}
 \comment{Lem. u}
\label{sec:orgheadline39}
\label{orgtarget33}
Let \(X \Tbein \Nb^\numcoins\), \(\sgY \Tbein \Nb\), and \(\ccR \Tbein \NN\).
Choose \(j \Tin \Omega\) so that \(x^j_\ccR \Tisge x^i_\ccR\) for every \(i \Tin \Omega\).
If \(\sgY_\ccR \Tisle \sum_{i} x^i_\ccR\), then
there exists \(u \Tin \range{\ybs[\ccR]}^\numcoins\) satisfying the following three conditions:
\begin{description}
\item[{(1)}] \(\displaystyle \bigoplus_{i \in \Omega} u^i \Teq \sgY_\ccR\), where \(u = (u^0, \ldots, u^{\numcoins - 1})\).
\item[{(2)}] \(0 \Tisle x_\ccR^i - \ccu^i \le x_\ccR^j - \ccu^j\) for every \(i \in \Omega\).
\item[{(3)}] \(\displaystyle \sum_{i \in \Omega} (x_\ccR^i - u^i) \Tisle \ybs[\ccR] - 1\).
\end{description}
 
\end{lemma}

\comment{connect}
\label{sec:orgheadline40}
Before proving Lemma \ref{orgtarget33}, let us give an example.
Let \(\ybs \Tbe (7,2,\ldots)\), \(X \Tbe (4, 4, 3)\), \(\ccR \Tbe j = 0\), and \(\sgY \Tbe 6\).
Then \(\sgY_0 \Teq 6 < 11 = \sum_i x_0^i\).
Since \((\bigoplus_i x_0^i) \ominus \sgY_0 \Teq 4 \ominus 6 = 5 < 11\), 
we can obtain \(u\) satisfying (1)--(3) by subtracting a vector \((t^0, t^1, t^2)\) with \(\sum t^i = 5\) from \(x_0\).
For example, let \(t \Tbe (4, 1, 0)\); then \(x_0 - t \Teq (0, 3, 3)\), and it satisfies (1)--(3).

\begin{proof}
 \comment{Proof.}
\label{sec:orgheadline41}
By rearranging \(x^0, \ldots, x^{\numcoins - 1}\) if necessary, we may assume that \(j \Teq 0\).
Let \(d \Tbe (\bigoplus_i x_\ccR^i) \ominus \sgY_\ccR\). Since \(\sgY_\ccR \Tisle \sum_{i} x^i_\ccR\), it follows that \(d \Tisle \sum_i x_\ccR^i - \sgY_\ccR \le \sum_i x_\ccR^i\).
For \(i \Tin \Omega\), let
\[
 t^{i} \Tbe \min \Set{x_\ccR^i, d - \sum_{h = 0}^{i - 1}  t^h}.
\]
Then \(\sum_i t^i \Teq d\).
It follows that \(x_\ccR - (t^0, \ldots, t^{\numcoins - 1})\) satisfies (1)--(3).
\end{proof}

\comment{Proof of A3}
\label{sec:orgheadline50}
We now prove (A3).
Let \(\sgX \Tbe \mphi(X)\) and \(\sgY \Tbein \Nb\) with \(0 \le \sgY < \sgX\).
We show that \(\mphi(X - C) \Teq \sgY\) for some \(C \Tin \sgCord[\ybs][[w]]\) with \(X - C \in \cPmis\).
Let \(\cMX \Tbe \mord_{\ybs}(X)\) and
\[
 \cMDiff \Tbe \max \set{L \in \NN : \sgY_L \neq \sgX_L}.
\]
We divide the proof into two cases.

\resetmycase
 
 \begin{mycase}[\(\cMX > \cMDiff\)]
 \comment{Case. [\(\cMX > \cMDiff\)]}
\label{sec:orgheadline42}
By rearranging \(x^0, \ldots, x^{\numcoins - 1}\) if necessary, we may assume that \(x^0_\cMX \Tneq 0\).
For \(i \Tge 1\), let \(y^i \Tbe x^i\).
Let
\[
 \cMY \Tbe \mord_{\ybs}((\sgY \ominus (-1), y^1, \ldots, y^{\numcoins - 1})),
\]
\[
 y^0 \Tbe \sgY \ominus (\ybsp[\cMY + 1]  - 1) \ominus y^1 \ominus \cdots \ominus y^{\numcoins - 1},
\]
and \(Y = (y^0, y^1, \ldots, y^{\numcoins - 1})\).
By Lemma \ref{orgtarget31}, \(\mphi(Y) \Teq \sgY\) and \(\mord_{\ybs}(Y) = \cMY\).
It remains to prove that \(y^0 \Tislt x^0\).
We show that \(y^0_\cMX \Teq x^0_\cMX - 1\) and \(y^0_{\ge \cMX + 1} \Teq x^0_{\ge \cMX + 1}\).
Since \(\sgY_\cMDiff \Tislt \sgX_\cMDiff \le \ybs[\cMDiff] - 1\),
we see that \((\sgY \ominus (-1))_\cMDiff \Teq \sgY_\cMDiff \oplus 1 \neq 0\). 
Hence \(\cMY \Tisle \cMDiff < \cMX\).
Since \(\mord_{\ybs}(Y) \Teq \cMY\),
\begin{equation}
\label{orgtarget34}
 \sgY_\cMX \Teq y^0_\cMX \oplus y^1_\cMX \oplus \cdots \oplus y^{\numcoins - 1}_\cMX \ominus \Compare{\cMX}{\cMY} = y^0_\cMX \oplus x^1_\cMX \oplus \cdots \oplus x^{\numcoins - 1}_\cMX.
\end{equation}
Moreover, since \(\cMX \Tisgt \cMDiff\),
\begin{equation}
\label{orgtarget35}
 \sgY_\cMX \Teq \sgX_\cMX = x^0_\cMX \oplus x^1_\cMX \oplus \cdots \oplus x^{\numcoins - 1}_\cMX \ominus \Compare{\cMX}{\cMX} = x^0_\cMX \oplus x^1_\cMX \oplus \cdots \oplus x^{\numcoins - 1}_\cMX \ominus 1.
\end{equation}
By (\ref{orgtarget34}) and (\ref{orgtarget35}), \(y^0_\cMX \Teq x^0_\cMX \ominus 1\).
Since \(x^0_\cMX \Tneq 0\), it follows that \(y^0_\cMX \Teq x^0_\cMX  \ominus 1 = x^0_\cMX - 1\).
Similarly, for \(L \Tge \cMX + 1\),
\[
 y^0_L \oplus x^1_L \oplus \cdots \oplus x^{\numcoins - 1}_L \ominus \Compare{L}{\cMY} \Teq \sgY_L = \sgX_L = x^0_L \oplus x^1_L \oplus \cdots \oplus x^{\numcoins - 1}_L \ominus \Compare{L}{\cMX}.
\]
Since \(\Compare{L}{\cMY} \Teq \Compare{L}{\cMX} = 0\), we see that \(y^0_L \Teq x^0_L\). 
Therefore \(y^0 \Tislt x^0\) and \(X - Y \Teq (x^0 - y^0, 0,\ldots, 0) \in \sgCord[\ybs][[w]]\).
 
\end{mycase}

 \begin{mycase}[\(\cMX \le \cMDiff\)]
 \comment{Case. [\(\cMX \le \cMDiff\)]}
\label{sec:orgheadline49}
By rearranging \(x^0, \ldots, x^{\numcoins - 1}\) if necessary, we may assume that \(x^0_{\cMDiff} \Tisge x^1_{\cMDiff} \ge \cdots \ge x^{\numcoins - 1}_{\cMDiff}\).
Let \(\cMY \Tbe \mord_{\ybs}((\sgY \ominus (-1), x^1, \ldots, x^{\numcoins - 1}))\). Since \(\sgY_\cMDiff \Tislt \sgX_\cMDiff \le \ybs[\cMDiff] - 1\), it follows that \(\cMY \Tisle \cMDiff\).

 \begin{myclaim}
 \comment{MyClaim. one-digit2}
\label{sec:orgheadline43}
\label{orgtarget36}
There exists \(\ccu \Tin \range{\ybs[\cMDiff]}^\numcoins\) satisfying the following four conditions:
\begin{description}
\item[{(C1)}] \(\displaystyle \bigoplus_{i \in \Omega} \ccu^i \Teq \sgY_{\cMDiff} \oplus \Compare{\cMDiff}{\cMY}\), where \(\ccu = (\ccu^0, \ldots, \ccu^{\numcoins - 1})\).
\item[{(C2)}] \(0 \Tisle x_{\cMDiff}^i - \ccu^i \le x_\cMDiff^0 - \ccu^0\) for every \(i \in \Omega\).
\item[{(C3)}] \(\ccu^0 \Tislt x^0_\cMDiff\) unless \(\cMX < \cMDiff = \cMY\).
\item[{(C4)}] \(\wt(x_\cMDiff - \ccu) \Tisle w\).
\end{description}
 
\end{myclaim} 

\comment{Construction of Y}
\label{sec:orgheadline44}
Assuming the claim for the moment, we construct \(Y\) with \(\mphi(Y) = \sgY\) and \(X - Y \in \sgCord[\ybs][[w]]\).
For \(i \Tge 1\), let
\begin{equation}
\label{orgtarget37}
 y^i \Tbe \bnums{x^i_0, \ldots, x^i_{\cMDiff - 1}, \ccu^i, x^i_{\cMDiff + 1}, x^i_{\cMDiff + 2}, \ldots} = x^i - (x^i_{\cMDiff} - u^i) \ybsp[\cMDiff] \text{\ (}\le x^i \text{)}.
\end{equation}
Then
\begin{equation}
\label{orgtarget38}
 \mord_{\ybs}((\sgY \ominus (-1), y^1, \ldots, y^{\numcoins - 1})) \Teq \cMY.
\end{equation}
Indeed, if \(\cMY \Tislt \cMDiff\), then (\ref{orgtarget38}) is obvious.
If \(\cMY \Teq \cMDiff\), then \((\sgY \ominus (-1))_\cMY  = (\sgY \ominus (-1))_\cMDiff \Tneq 0\), so (\ref{orgtarget38}) holds.
Let
\begin{equation}
\label{orgtarget39}
 y^0 \Tbe \sgY \ominus (\ybsp[\cMY + 1] - 1) \ominus y^1 \ominus \cdots \ominus y^{\numcoins - 1}
\end{equation}
and \(Y = (y^0, y^1, \ldots, y^{\numcoins - 1})\).
It follows from (\ref{orgtarget38}) and Lemma \ref{orgtarget31} that \(\phi(Y) \Teq \sgY\) and \(\mord_{\ybs}(Y) \Teq \cMY\).

\comment{Verification of C in \sgCord}
\label{sec:orgheadline45}
Let \(C \Tbe X - Y\).
We next show that \(C \Tisin \sgCord[\ybs][[w]]\), that is,
\begin{description}
\item[{(a)}] \(C \Tisin \Nb^\numcoins \setminus \set{(0,\ldots,0)}\),
\item[{(b)}] \(\wt(C) \Tisle w\), and
\item[{(c)}] \(\ord_{\ybs}(\sum c^i) \Teq \mord_{\ybs}(C)\).
\end{description}

(a) Since \(\phi(Y) \Teq \sgY \neq \sgX = \phi(X)\), it follows that \(C \Tneq (0, \ldots, 0)\). By (\ref{orgtarget37}), we see that \(y^i \Tisle x^i\) for \(i \Tge 1\).
To prove that \(y^0 \Tisle x^0\), we show that \(y^0_{\ge \cMDiff + 1} \Teq x^0_{\ge \cMDiff + 1}\) and \(y^0_\cMDiff \Teq u^0\). By (\ref{orgtarget39}), for \(L \Tgt \cMDiff (\ge \cMY, \cMX)\),
\[
 y^0_L \Teq \sgY_L \oplus \Compare{L}{\cMY} \ominus y^1_L \ominus \cdots \ominus y^{\numcoins - 1}_L =  \sgX_L \ominus  x^1_L \ominus \cdots \ominus x^{\numcoins - 1}_L = x^0_L \ominus [L \le \cMX] = x^0_L.
\]
We also see that
\[
 y^0_\cMDiff \Teq \sgY_\cMDiff \oplus \Compare{\cMDiff}{\cMY} \ominus y^1_\cMDiff \ominus \cdots \ominus y^{\numcoins - 1}_\cMDiff =  \sgY_\cMDiff \oplus \Compare{\cMDiff}{\cMY} \ominus u^1 \ominus \cdots \ominus u^{\numcoins - 1} = u^0
\]
by (C1).
Hence if \(u^0 \Tislt x^0_\cMDiff\), then \(y^0 \Tislt x^0\).
Suppose that \(u^0 \Teq x^0_\cMDiff\). Then \(\cMY \Teq \cMDiff\) by (C3).
Since \(\cMY \Teq \mord_{\ybs}(Y)\),
\[
 y^0 \Teq \bnums{\underbrace{0, \ldots, 0}_{\cMY}, \ccu^0, x^0_{\cMY + 1}, x^0_{\cMY + 2}, \ldots} = \bnums{\underbrace{0, \ldots, 0}_{\cMY}, x^0_\cMY, x^0_{\cMY + 1}, x^0_{\cMY + 2}, \ldots} \le x^0.
\]

(b) If \(\ccu^0 \Tneq x^0_\cMDiff\), then \(\wt(C) \Teq \wt(x_{\cMDiff} - \ccu)\), and hence \(\wt(C) \Tisle w\) by (C4).
Suppose that \(\ccu^0 \Teq x^0_\cMDiff\).
By (C2), \(\ccu \Teq x_\cMDiff\), so \(\wt(C) \Teq 1 \le w\).

(c) For \(i \Tge 1\), we know that \(c^i \Teq x^i - y^i = (x^i_\cMDiff - u^i) \ybsp[\cMDiff]\). 
Hence if \(\ord_{\ybs}(c^0) \Tislt \cMDiff\), then 
\[
 \mord_{\ybs}(C) \Teq \ord_{\ybs}(c^0) = \ord_{\ybs}\left(\sum_{i \in \Omega} c^i\right).
\]
Suppose that \(\ord_{\ybs}(c^0) \Tisge \cMDiff\). Then \(\cMDiff \Tisle \mord_{\ybs}(C) \le \ord_{\ybs}(\sum c^i)\),
so we need only show that \(\ord_{\ybs}(\sum c^i) \Teq \cMDiff\), that is, 
\begin{equation}
\label{orgtarget40}
 \bigoplus_{i \in \Omega} c^i_\cMDiff  \Tneq 0.
\end{equation}
We first show that \(\cMY \Teq \cMX\).
Since \(Y_{< \cMDiff} \Teq X_{< \cMDiff}\) and \(\cMY, \cMX \Tisle \cMDiff\),
it follows that if \(\cMY \Tislt \cMDiff\) or \(\cMX \Tislt \cMDiff\), then \(\cMY \Teq \cMX\).
If \(\cMY \Tisge \cMDiff\) and \(\cMX \Tisge \cMDiff\), then \(\cMY \Teq \cMDiff = \cMX\).
We now show (\ref{orgtarget40}).
Since \(y^i_\cMDiff \Teq x^i_\cMDiff \ominus c^i_\cMDiff\), we see that
\begin{align*}
 \sgY_\cMDiff \Teq \bigoplus_{i \in \Omega} y^i_\cMDiff  \ominus \Compare{\cMDiff}{\cMY} &= \bigoplus_{i \in \Omega} x^i_\cMDiff  \ominus \left( \bigoplus_{i \in \Omega} c^i_\cMDiff \right) \ominus \Compare{\cMDiff}{\cMX} \\
 &= \sgX_\cMDiff \ominus \left( \bigoplus_{i \in \Omega} c^i_\cMDiff \right).
\end{align*}
Since \(\sgY_\cMDiff \Tneq \sgX_\cMDiff\), it follows that (\ref{orgtarget40}) holds, so \(\ord_{\ybs}(\sum c^i) \Teq \cMDiff = \mord_{\ybs}(C)\).
Therefore \(C \Tisin \sgCord[\ybs][[w]]\).

\comment{Proof of claim}
\label{sec:orgheadline46}
It remains to prove the claim.
We first show that \(x_\cMDiff \Tneq (0, \ldots, 0)\).
If \(\cMDiff \Teq \cMX\), then the assertion is obvious.
If \(\cMDiff \Tisgt \cMX\), then  \(\Compare{\cMDiff}{\cMX} \Teq 0\), so \(\sgY_\cMDiff \Tislt \sgX_\cMDiff = \bigoplus_i x_\cMDiff^i\), which implies that \(x_\cMDiff \Tneq (0, \ldots, 0)\).

We next show that
\begin{equation}
\label{orgtarget41}
\sgX_{\cMDiff} + \Compare{\cMDiff}{\cMX} \Tisle \sum_{i \in \Omega} x^i_{\cMDiff}.
\end{equation}
Since \(\sgX_\cMDiff \oplus \Compare{\cMDiff}{\cMX} \Teq \bigoplus_i x^i_\cMDiff \le \sum_i x^i_\cMDiff\),
if \(\sgX_\cMDiff + \Compare{\cMDiff}{\cMX} \Teq \sgX_\cMDiff \oplus \Compare{\cMDiff}{\cMX}\), then (\ref{orgtarget41}) holds.
Suppose that \(\sgX_\cMDiff + \Compare{\cMDiff}{\cMX} \Tneq \sgX_\cMDiff \oplus \Compare{\cMDiff}{\cMX}\).
Then \(\Compare{\cMDiff}{\cMX} \Teq 1\) and \(\sgX_\cMDiff \oplus \Compare{\cMDiff}{\cMX} \Teq 0\).
Hence \(\bigoplus_i x^i_\cMDiff \Teq 0\).
Since \(x_\cMDiff \Tneq (0, \ldots, 0)\), it follows that \(\sum x^i_\cMDiff \Tisge \ybs[\cMDiff] = \sgX_\cMDiff + \Compare{\cMDiff}{\cMX}\).

We now construct \(\ccu\) satisfying (C1)--(C4).
Since \(\Compare{\cMDiff}{\cMY}, \Compare{\cMDiff}{\cMX} \Tisin \set{0, 1}\) and \(\sgY_{\cMDiff} \Tislt \sgX_{\cMDiff}\), we see that
\[
 \sgY_{\cMDiff} \oplus \Compare{\cMDiff}{\cMY} \Teq \sgY_\cMDiff + \Compare{\cMDiff}{\cMY} \le \sgX_\cMDiff + \Compare{\cMDiff}{\cMX} \le \sum_{i \in \Omega} x^i_{\cMDiff}.
\]
By Lemma \ref{orgtarget33}, there exists \(\ccu \Tin \range{\ybs[\cMDiff]}^\numcoins\) satisfying (C1), (C2), and
\[
 \wt(x_\cMDiff - u) \Tisle \min \set{\ybs[\cMDiff] - 1, \numcoins}.
\]
 We divide the proof into two cases.

\resetmycaseroman
 
 \begin{mycaseroman}[\(\cMDiff > 0\)]
 \comment{CaseRoman. [\(\cMDiff > 0\)]}
\label{sec:orgheadline47}
Since \(\min \set{\ybs[\cMDiff] - 1, \numcoins} \Tisle w\), we see that \(\ccu\) satisfies (C4).
If \(\ccu \Tneq x_\cMDiff\), then \(\ccu^0 \Tislt x_\cMDiff^0\), so \(u\) also satisfies (C3).
Suppose that \(\ccu \Teq x_\cMDiff\). Then
\begin{equation}
\label{orgtarget42}
 \sgY_\cMDiff \oplus \Compare{\cMDiff}{\cMY} \Teq \bigoplus_{i \in \Omega} \ccu^i = \bigoplus_{i \in \Omega}x^i_{\cMDiff} = \sgX_\cMDiff \oplus \Compare{\cMDiff}{\cMX}.
\end{equation}
It follows that \(\Compare{\cMDiff}{\cMY} \Tneq \Compare{\cMDiff}{\cMX}\) since \(\sgY_\cMDiff \Tislt \sgX_\cMDiff\). 
Suppose that \(\Compare{\cMDiff}{\cMY} \Teq 1\) and \(\Compare{\cMDiff}{\cMX} \Teq 0\).
Then \(\cMDiff \Tisle \cMY\) and \(\cMDiff \Tisgt \cMX\), so \(\cMX \Tislt \cMDiff = \cMY\) since \(\cMY \Tisle \cMDiff\). Hence \(u\) satisfies (C3).
Suppose that \(\Compare{\cMDiff}{\cMY} \Teq 0\) and \(\Compare{\cMDiff}{\cMX} \Teq 1\).
By (\ref{orgtarget42}), 
\[
 \sgY_\cMDiff \Teq \bigoplus_{i \in \Omega} x^i_\cMDiff = \sgX_\cMDiff \oplus 1 < \sgX_\cMDiff.
\]
This implies that \(\bigoplus_i x^i_\cMDiff \Teq 0\),
and hence that \(\sum_i x_\cMDiff^i \Tisge \ybs[\cMDiff]\) since \(x_\cMDiff \Tneq (0, \ldots, 0)\). 
We can now find \((\tilde{u}^0, \ldots, \tilde{u}^{\numcoins - 1})\) satisfying \(\sum_i \tilde{\ccu}^i \Teq \sum_i x^i_\cMDiff - \ybs[\cMDiff]\) and (C1)--(C3) in the same way as in the proof of Lemma \ref{orgtarget33}.
Let \(\tilde{u} \Teq (\tilde{u}^0, \ldots, \tilde{u}^{\numcoins - 1})\).
Since \(\cMDiff \Tisgt 0\),
\[
 \wt(x_\cMDiff - \tilde{\ccu}) \Tisle \min \set{\ybs[\cMDiff], \numcoins}  \le w,
\]
so \(\tilde{u}\) also satisfies (C4). 
 
\end{mycaseroman}

 \begin{mycaseroman}[\(\cMDiff = 0\)]
 \comment{CaseRoman. [\(\cMDiff = 0\)]}
\label{sec:orgheadline48}
We see that \(x_0 \Tneq \ccu\) because
\[
 \bigoplus_{i \in \Omega} x^i_0 \Teq \sgX_0 \oplus [0 \le \cMX] = \sgX_0 \oplus 1
\]
and
\[
 \bigoplus_{i \in \Omega} u^i \Teq \sgY_0 \oplus [0 \le \cMY] = \sgY_0 \oplus 1.
\]
Thus \(u\) satisfies (C3).
Since \(\wt(x_0 - u) \Tisle \min \set{\ybs[0] - 1, \numcoins}\),
we see that (C4) holds when \(\min \set{\ybs[0] - 1, \numcoins - [\ybs[0] < 2\numcoins]} \Teq \ybs[0] - 1\) or \(\numcoins\).
Suppose that 
\[
 \min \set{\ybs[0] - 1, \numcoins - [\ybs[0] < 2\numcoins]} \Teq \numcoins - 1 < \ybs[0] - 1.
\]
Then \(\numcoins \Tislt \ybs[0] < 2\numcoins\) and \(\numcoins - 1 \le w\). 
We show that there exists \(\tilde{\ccu}\) satisfying (C1)--(C4).
If \(x^{\numcoins - 1}_0 \Teq 0\), then \(\ccu^{\numcoins - 1} = 0\), so \(\ccu\) itself satisfies (C1)--(C4).
Suppose that \(x^{\numcoins - 1}_0 \Tisge 1\). We show that
\begin{equation}
\label{orgtarget43}
 x_0^0 + \cdots + x^{\numcoins - 2}_0 \Tisge \sgY_0 \oplus 1 \ominus x_0^{\numcoins - 1}.
\end{equation}
If \(x_0^{\numcoins - 1} \Teq 1\), then \(\sgX_0 \Teq (\bigoplus_i x^i_0) \ominus 1 = x^0_0 \oplus \cdots \oplus x^{\numcoins - 2}_0\), and hence
\[
 x_0^0 + \cdots + x_0^{\numcoins - 2} \Tisge x_0^0 \oplus \cdots \oplus x^{\numcoins - 2}_0 =  \sgX_0  >  \sgY_0 = \sgY_0 \oplus 1 \ominus x_0^{\numcoins - 1}.
\]
If \(x_0^{\numcoins - 1} \Tisge 2\), then, since \(x^0_0 \Tisge x^1_0 \ge \cdots \ge x^{\numcoins - 1}_0 \ge 2\) and \(2\numcoins \Tisgt \ybs[0]\),
\[
 x_0^0 + \cdots + x_0^{\numcoins - 2} \Tisge 2 (\numcoins - 1) \ge  \ybs[0] - 1 \ge \sgY_0 \oplus 1 \ominus x_0^{\numcoins - 1}.
\]
Thus (\ref{orgtarget43}) holds.
By Lemma \ref{orgtarget33}, there exists \((\tilde{\ccu}^0, \ldots, \tilde{\ccu}^{\numcoins - 2})\) such that \(\tilde{\ccu}^0 \oplus \cdots \oplus \tilde{\ccu}^{\numcoins - 2} \Teq \sgY_0 \oplus 1 \ominus x_0^{\numcoins - 1}\) and \(0 \le x_0^i - \tilde{\ccu}^i \Tisle x_0^0 - \tilde{\ccu}^0\).
Thus \((\tilde{u}^0, \ldots, \tilde{u}^{\numcoins - 2}, x_0^{\numcoins - 1})\) satisfies (C1)--(C4).
This completes the proof.
\qed
 
\end{mycaseroman} 
 
\end{mycase} 

\comment{bibliography}
\label{sec:orgheadline51}
\bibliographystyle{abbrv}

\end{document}